\documentclass[10pt]{amsart}
\usepackage{geometry}                
\usepackage{graphicx}
\usepackage{amssymb}
\usepackage{epstopdf}
\usepackage{parskip}
\usepackage{amsthm}
\usepackage{multirow}
\usepackage[toc,page]{appendix}
\usepackage{float}
\usepackage{amsmath}
\usepackage{natbib}
\usepackage{url}

\newtheorem{theorem}{Theorem}
\newtheorem{lemma}{Lemma}
\newtheorem{corollary}{Corollary}

\title{Fermat's Last Theorem: Algebra and Number Theory}
\author{Felix Sidokhine}
\date{}                                           

\begin{document}
\maketitle

\begin{abstract}
In our work we give the examples using Fermat's Last Theorem for solving some problems from algebra and number theory.
\end{abstract}

\section{Introduction}

The proof of Fermat's last theorem is viewed as one of the crown accomplishments in mathematics. However, after the orchestra faded the community was left with questions that most of us try to avoid - how we can apply Fermat's theorem to obtain new proofs to some of the known theorems in algebra, number theory and geometry as well as derive some new results. Such a work carries value, even though one could argue that it is educational rather than scientific, as we known that it is possible to prove theorems by using ``hard'' ways and ``easy'' ways. But if Fermat's Last Theorem is true, then why should ignore its possible applications to problems of algebra, number theory and geometry?

\section{Fermat's Theorem \&  Algebra}

One of the possible generalizations of Fermat's last theorem is the Euler-Ekel hypothesis. In this section we will discuss the connection between these two objects by looking at them through the prism of polynomials and splitting fields. In this section we are working with integral polynomials over the field of rational numbers $\mathbb{Q}$.

Let us first take a look at a very simple, yet elegant theorem:

\begin{theorem}\label{thm00}
Let $p(x)=x^3+bx+a^n$ where $a,b \text{ } (a > 0, b \neq 0)$ are co-prime.

If $\mathbb{Q}$ is a splitting field for $p(x)$, then there exist $p, q, r \in \mathbb{Z}^+$  such that $a = pqr$ and $(p, q, r)$ is a solution of the equation $X^n+Y^n=Z^n$ where $X, Y, Z$ are pairwise co-prime.

Conversely, if there exist such positive integers $p, q, r$ which are a solution of $X^n+Y^n=Z^n$, where  $X, Y, Z$ are pairwise co-prime, then there exists a polynomial $p(x)=x^3+bx+a^n$, where $a = pqr$, $b \neq 0$, $\gcd(a,b)=1$ such that $\mathbb{Q}$ is its splitting field.
\end{theorem}

\begin{proof}
Let the conditions of theorem \ref{thm00} be satisfied and a polynomial $p(x)$ is a product of the linear factors over $\mathbb{Q}$ then $p(x) = (x - \alpha)( x - \beta)( x + \gamma)$, where $a^n= \alpha\beta\gamma$ and $\alpha, \beta, \gamma$ are pairwise relatively prime positive integers. Hence $\alpha=p^n, \beta=q^n, \gamma=r^n$ and $p^n+q^n=r^n$.

Let positive integers $p, q, r$ be a solution of the equation $X^n+Y^n=Z^n$, where $X, Y, Z$ are co-prime in pairs then we can construct the polynomial $p(x)=(x -p^n)(x -q^n)(x+r^n)$.
\end{proof}
An interesting corollary which follows from this theorem is:

\begin{corollary}
Any cubic polynomial of the form $p(x)=x^3+bx+a^n$, where $a, b$ $(ab \neq 0)$ are co-prime and $n \geq 3$  over the field $\mathbb{Q}$ is either irreducible or a product of two irreducible polynomials.
\end{corollary}

We shall omit the proof for simplicity's sake as it is a direct application of theorem \ref{thm00} and Fermat's last theorem.

Now let us consider a more general case where we look at the polynomial of arbitrary degree $n$: $p(x) = x^n+ a_{n-2} x^{n-2}+ ... + a_1 x \pm a_0$ where $a_1 a_0 \neq 0$, $\gcd(a_1,a_0)=1$ and $a_0=c^k$ for any $c \in \mathbb{Z}^+$ .

Let $\mathbb{Q}$ be a splitting field for $p(x)$ and $\alpha_1, \alpha_2, ..., \alpha_h \in \mathbb{Z}^+$,  $-\beta_1,-\beta_2,...,-\beta_l \in \mathbb{Z}^-$ are the roots of $p(x)$ and $n = h + l$. Then, $c^k =  \alpha_1...\alpha_h \beta_1...\beta_l$. Moreover, since $\gcd(a_0,a_1)=1$, we can claim that $\alpha_1, \alpha_2,...,\alpha_h, \beta_1, \beta_2,...\beta_l$ are all pairwise co-prime and that  $\alpha_1=x_1^k, \alpha_2=x_2^k, ..., \alpha_h=x_h^k, \beta_1=y_1^k  , \beta_2=y_2^k, ..., \beta_l=y_l^k$ for some $x_i$'s and $y_i$'s.

Therefore we have that:
\begin{equation*}
x_1^k  + x_2^k  + ... +x_h^k  - y_1^k  - y_2^k  - ... - y_l^k  = 0 \text{ where } n = h + l
\end{equation*} 

must be true. However, we can invoke the Euler - Ekl hypothesis (1769, 1998) which states that  ``The equation $x_1^k  + x_2^k  + ... +x_h^k  - y_1^k  - y_2^k  - ... - y_l^k  = 0$, has no solution in positive integers when $k>h+l$'' \cite{Ekl1998}. From the above, this allows us to conclude that there exists no $p(x) = x^n+ a_{n-2} x^{n-2}+ ... + a_1 x \pm a_0$ where $a_1 a_0 \neq 0$, $\gcd(a_1,a_0)=1$, $a_0=c^k$ and $k > n$ for which $\mathbb{Q}$ is a splitting field.

A ``weak conjecture'' we can make is: ``The equation  $x_2^k  + ... +x_h^k  - y_1^k  - y_2^k  - ... - y_l^k  = 0$, where $x_1,..., x_h, y_1,..., y_l$ are co-prime in pairs, has no solution in positive integers when $k>h+l$'', for which can formulate an interesting theorem:

\begin{theorem}
The ``weak conjecture'' is false if and only if there exists a polynomial $p(x)=x^n+a_{n-2} x^{n-2}+ ... + a_1 x \pm a_0$, where $a_0=c^k, c > 0, a_1 a_0 \neq 0$ and $a_1,a_0$ are relatively prime, such that the field $\mathbb{Q}$ is its splitting field.
\end{theorem}

\begin{proof}
If $\mathbb{Q}$ is the splitting field for some polynomial $p(x)=x^n+a_{n-2} x^{n-2}+ ... + a_1 x\pm a_0$ then the equation described above has a solution. If the equation has a solution then we can construct  a polynomial $p(x)$ that the field $\mathbb{Q}$ is its splitting field.
\end{proof}

A  proof of the ``weak conjecture'' itself might be possible by induction. In this case, we will define the inductive hypothesis as: ``The field $\mathbb{Q}$ is not a splitting field for any polynomial of the form $p(x)=x^n+a_{n-2} x^{n-2}+ ... + a_1 x\pm a_0$, where $a_0=c^k$, $k>n$, $a_1 a_0 \neq 0$ and $\gcd(a_1,a_0) = 1$, then $\mathbb{Q}$ is not a splitting field for any polynomial form $g(x)=x^n \pm b_{n-1} x^{n-1}+ ... + b_1 x \pm b_0$, where $b_{n-1} b_0= c^k$, $k > n+1$,  $b_{n-1} b_1 b_0 \neq 0$, and $\gcd(b_{n-1} b_1 ,b_0) = 1$''.

A sketch of the proof would then go as follows:

\begin{itemize}

\item[Step 1]: For the case $n = 3$, $p(x) = x^3+a_1 x \pm a_0$, where $a_0=c^k$, $k>3$, $a_1 a_0 \neq 0$ and $\gcd(a_1,a_0) = 1$. We have already shown that $\mathbb{Q}$ is not a splitting field for such polynomials.

\item[Step 2]: For the case $n = 3$, $g(x)=x^3 \pm b_2 x^2+ b_1 x \pm b_0$, where $b_2 b_0= c^k$, $k>4$, $b_2 b_1 b_0 \neq 0$, and $\gcd(b_2 b_1 ,b_0) = 1$, the field $\mathbb{Q}$ is not a splitting field for such polynomials as a consequence of the inductive hypothesis.

\item[Step 3]: For the case $n = 4$, $p(x)=x^4+a_2 x^2+a_1 x \pm a_0$, where $a_0=c^k$, $k>4$, $a_1 a_0 \neq 0$ and $\gcd(a_1,a_0) = 1$, the field $\mathbb{Q}$ is not a splitting field by step two.

\item[Step 4]: For the case $n = 4$, $g(x)=x^4 \pm b_3 x^3+b_2 x^2+ b_1 x \pm b_0$, where $b_3 b_0= c^k$, $k>5$, $b_3 b_1 b_0 \neq 0$, and $\gcd(b_3 b_1 ,b_0) = 1$, the field $\mathbb{Q}$ is not a splitting field as a consequence of the inductive hypothesis. 

\item[Step 5]: For the case $n = 5$, $p(x)=x^5+a_3 x^3+a_2 x^2+a_1 x \pm a_0$, where $a_0=c^k$, $k>5$, $a_1 a_0 \neq 0$ and $\gcd (a_1,a_0) = 1$, the field $\mathbb{Q}$ is not a splitting field by step four.

\end{itemize}

and the inductive sequence continues. 

The existing counter-examples to Euler's conjecture do not disprove our suggested alternative conjecture :``The equation $x_1^k+x_2^k  + ... +x_n^k  - y^k= 0$, where $x_1, x_2,…, x_h, y$ are relatively prime in pairs, has no solution in positive integers when $k > n \geq 2$''.

All the existing counter examples are connected to the polynomials of the form $p_4 (x)=x^4+a_2 x^2+a_1 x - a_0^4$ and $p_5 (x)=x^5+b_3 x^3+b_2 x^2+b_1 x+b_0^5$ where $a_0, a_1$ and $b_0, b_1$ are not co-prime and for which $\mathbb{Q}$ becomes a splitting field. In other words, the known solutions of equations $x_1^4  + x_2^4  + x_3^4  = y^4$,  $x_1^5  + x_2^5  + x_3^5  + x_4^5= y^5$ in positive integers satisfy do not satisfy the condition that $x_1,..., x_h, y$ should be pairwise co-prime (see appendix for the concrete examples which are taken from \cite{Elkies1998}, \cite{LanderParkin1966} and \cite{MalterSchleicherZagier2013}).

\section{Fermat's Theorem \& Number Theory}

The goal of the present section is to study the relationship between the unsolvability of some Diophantine equations and Fermat's last theorem. 

\begin{theorem}\label{thm0}
The equation $x^n + y^n + z^n = u^n$ with $xy = zu$, where $\gcd(x, y)=\gcd(z, u)=1$ has no solution over $\mathbb{Z}^+$ when $n \geq 2$.
\end{theorem}

\begin{lemma}\label{lem0}
Let us consider the following system of equations over $\mathbb{Z}^+$
\begin{equation}
\begin{cases}\label{sys0}
X^n + Y^n = {X'}^n - {Y'}^n \\
XY = X'Y',
\end{cases}
\end{equation}
where $\gcd(X, Y)= \gcd(X', Y')=1$. If the system (\ref{sys0}) is solvable, then $XY \equiv 0 \mod 2$.
\end{lemma}

\begin{proof}
Let:
\begin{equation*}
\begin{cases}
X = a, Y = b, X' = a', Y' = b' \\
\gcd(a, b) = \gcd(a', b') = 1
\end{cases}
\end{equation*}
 be a solution of (\ref{sys0}) and $ab \equiv 1 \mod 2$. Then $a$ and $b$ are products of primes of the form $4k + 1$ and $4k' - 1$. 
 \begin{itemize}
 	\item  Let $ab$ be a product of primes only of the form $4k + 1$. Then $a^n + b^n \not \equiv {a'}^n - {b'}^n \mod 4$ and $a^n + b^n = {a'}^n - {b'}^n$ is false. 
	
	\item Let $ab$ be a product of primes only of the form $4k' - 1$ and the total number of such primes is even. Let $a$ contain an even number of primes and $a'$ contain an odd number of primes; then $b$ contains an even number of primes and $b'$ contains an odd number of primes. This would imply that $a^n + b^n \not \equiv {a'}^n - {b'}^n \mod 4$ and the equality $a^n + b^n = {a'}^n - {b'}^n$ would be false. The remaining cases are excluded analogously.
	
	\item Let $ab$ be a product of primes only of the form $4k' - 1$ and the total number of such primes is odd. Let $a$ contain an even number of primes and $a'$ contain an odd number of primes; then $b$ contains an odd number of primes and $b'$ contains an even number of primes. This would imply that $a^n + b^n \not \equiv {a'}^n - {b'}^n \mod 4$ and the equality $a^n + b^n = {a'}^n - {b'}^n$ would be false. The remaining cases are excluded analogously.
 \end{itemize}
\end{proof}

\begin{lemma}\label{lem1}
The following system of equations over $\mathbb{Z}^+$
\begin{equation}\label{sys1}
\begin{cases}
X^n + Y^n = {X'}^n - {Y'}^n \\
XY = X'Y'
\end{cases}
\end{equation}
where $\gcd(X, Y)=\gcd(X', Y')=1$ and $XY \neq 0$, has no solutions when $n \geq 2$.
\end{lemma}

\begin{proof}
Let $a, b, a', b'$ be a solution of the system (\ref{sys1}) where $\gcd(a, b) = \gcd(a', b') = 1$. For the following equations assume the notation $\gcd(a,b) = (a,b)$. According to lemma \ref{lem0}, $ab = a'b' \equiv 0 \mod 2$. Let $a \equiv a' \mod 2$ then we can write the equalities:
\begin{equation*}
a = (a, a')(a, b'),  b = (b, a')(b, b'); a'= (a', a)(a', b),  b' = (b', a)(b', b)
\end{equation*}
and
\begin{eqnarray*}
(a, a')^n (a, b')^n + (b, a')^n (b, b')^n = (a', a)^n (a', b)^n - (b', a)^n (b', b)^n \\
(a, a')^n(a, b')^n + (b', a)^n (b', b)^n = (a', a)^n (a', b)^n - (b, a')^n(b, b')^n
\end{eqnarray*}
or
\begin{equation*}
(a, b')^n ((a, a')^n + (b', b)^n) = (a', b)^n ((a', a)^n - (b, b')^n)
\end{equation*}
\newline	
since $\gcd ((a, b'), (a', b)) = \gcd((a, a')^n + (b', b)^n, (a', a)^n - (b, b')^n) = 1$
Thus we obtained the following equalities
\begin{equation*}	
(a, b')^n = (a', a)^n - (b, b')^n, (a', b)^n = (a, a')^n + (b', b)^n
\end{equation*}
or 				
\begin{equation*}
((a', b)(a, b'))^n = (a', a)^{2n} - (b, b')^{2n}
\end{equation*}
This equality is false for any $n \geq 2$ due to Fermat's theorem. Thus the lemma is true.
\end{proof}

\begin{proof}[Proof of Theorem \ref{thm0}] Theorem \ref{thm0} is true due to lemmas \ref{lem0}, \ref{lem1}. \end{proof} 

\begin{corollary}
Let $a, b, c, n \in \mathbb{Z}^+$ and $ab = c, \gcd(a, b) = 1$ if $ab \equiv 1 \mod 2$ then the polynomials $p(x) = x^2 + (a^n + b^n) x - c^n$ are irreducible over $\mathbb{\mathbb{Q}}$ for  $n \geq 1$; If $ab \equiv 0 \mod 2$ then the polynomials $p(x) = x^2 + (a^n + b^n) x - c^n$ are irreducible over $\mathbb{\mathbb{Q}}$ for $n \geq 2$.
\end{corollary}

\begin{theorem}
The system of the equations over $\mathbb{Z}$:
\begin{equation}\label{sys2}
\begin{cases}
x_1^3  +  x_2^3 + x_3^3 + 3x_4^n = 0 \\
(x_1 + x_2 + x_3) x_4= 0,
\end{cases}
\end{equation}
where $x_1, x_2, x_3$ are co-prime, $x_1x_2x_3\neq 0$, has no solutions in $\mathbb{Z}$ when $n > 2$.
\end{theorem}

\begin{proof}
System (\ref{sys2}) can be reduced to: ``the equation $x_1x_2(x_1 + x_2) = x_4^n$, where $x_1, x_2$ are co-prime, has no solution in integers when $n > 2$'' which is equivalent to Fermat's last theorem.
\end{proof}

The equation $x_1x_2(x_1 + x_2) = x_3$ in rational integers, where $x_1, x_2$ are co-prime and $x_1x_2x_3\neq 0$, can be used for formulating various theorems and hypotheses, for example:
\begin{itemize}
	\item The equations $x_1x_2(x_1 + x_2) = x_3^4$ and $x_1x_2(x_1^2 + x_2^2) = x_3^2$, where $x_1, x_2$ are co-prime, have no solutions over $\mathbb{Z}$ and $\mathbb{Z}[i]$.
	
	\item Fermat's last theorem: the equation $x_1x_2(x_1+ x_2) = x_3^n$, where $x_1, x_2$ are relatively prime, has no solution over $\mathbb{Z}^+$ when $n > 2$. This equation was considered in the paper \cite{TCai}
	
	\item Euler's conjecture: the equation $x_1x_2x_3(x_1 + x_2 + x_3) = x_4^n$, where  $x_1, x_2,  x_3 , x_1+ x_2 + x_3$ are relatively prime in pairs, has no solution over $\mathbb{Z}^+$ when $n > 3$. 
\end{itemize}

Thus, Euler's conjecture: ``The equation  $x^n  + y^n  + z^n  = u^n$, where $x, y, z, u$ are pairwise relatively prime, has no solution in positive integers when $n > 3$'' is a natural sequential extension of Fermat's Last Theorem.

\section{Conclusion}

The main result of this paper are as follows:

\begin{itemize}
  \item ``The equation $x_1^k+x_2^k  + ...+x_n^k  - y^k= 0$, where $x_1,..., x_n, y$ are pairwise relatively prime, has no solutions in positive integers when $k>n \geq 2'$'.

 \item ``The equation $x_1^k+x_2^k+...+x_h^k-y_1^k-y_2^k-...-y_l^k=0$, where $x_1,..., x_h, y_1,..., y_l$   are pairwise relatively prime and $h  \geq  l$, has no solution in positive integers when $k>h+l \geq 3$''.

 \item  ``The equation $x^n+y^n+z^n=u^n$ with $xy = zu$, where $x, y$ and also $z, u$ are relatively prime, respectively, has no solution in natural numbers when $n  \geq  2$''.

\item ``The equation $x^n + y^n + z^n = u^n$, where $x,y,z,u$ are pairwise relatively prime, has no solution over $\mathbb{Z}^+$ when $n \geq 3$''.

\end{itemize}

\begin{appendices}
Counter-examples to Euler's conjecture (1769) 
``The equation $x_1^k+x_2^k+…+x_n^k=y^k$ has no solution in positive integers when $k > n \geq 2$''.

($n = 3, k = 4$)
\begin{eqnarray*}
2682440^4       +5365639^4     +18796760^4   =20615673^4,    \text{ Elkies (1988)}\\
95800^4           +217519^4        +414560^4        =422481^4,      \text{ R. Frye (1988)}\\ 
630662624^4+275156240^4+219076465^4=638523249^4,  \text{ MacLeod   (1997)}\\
1705575^4      +5507880^4      +8332208^4      =8707481^4,   \text{ Bernstein  (2001)} 
\end{eqnarray*}

($n = 4, k = 5$)
\begin{eqnarray*}
27^5+84^5      +10^5         +133^5     =144^5, \text{ Lander, Parkin  (1966)} \\
55^5+3183^5+28969^5+85282^5=85359^5,  \text{ J. Frye (2004)}
\end{eqnarray*}

\end{appendices}
\bibliography{references}
\bibliographystyle{plainnat}

\end{document}